\documentclass[12pt, 14paper,reqno]{amsart}
\usepackage{amsmath,amsfonts,amssymb}
\usepackage[breaklinks]{hyperref}
\usepackage{graphicx}
\usepackage{longtable}
\usepackage{array}
\usepackage{caption}
\usepackage{xcolor}
\makeatletter
\@namedef{subjclassname@2020}{%
  \textup{2020} Mathematics Subject Classification}
\makeatother

\theoremstyle{plain}
\theoremstyle{plain}

\newtheorem{theorem}{Theorem}
\newtheorem{lemma}{Lemma}

\numberwithin{equation}{section}

\newtheorem{thma}{Theorem}

\theoremstyle{proof}

\numberwithin{equation}{section}

\begin{document} 

\title[On the class number of certain cyclotomic fields] {On the class number of certain cyclotomic fields}
\author{K. Chakraborty and P. Rao}
\address{K. Chakraborty, SRM University-AP, Amaravati 522240, Andhra Pradesh, India}.
\email{kalyan.c@srmap.edu.in}
\address{P. Rao, SRM University-AP, Amaravati 522240, Andhra Pradesh, India}.
\email{pratik\_rao@srmap.edu.in}
\maketitle
\begin{abstract} 
We construct four infinite families of cyclotomic fields and show that each member of these families has class number a multiple of $3$. In fact, we show that the class number of the corresponding maximal real subfields has class number divisible by $3$. A construction by Kishi and Miyake \cite{MI} and a result of Yamaguchi \cite{YA} help us achieve this target. Finally, we produce computational evidence corroborating our results.
\end{abstract}

\maketitle
\section{Introduction } 
The study of class numbers of number fields occupies a central role in algebraic number theory. The divisibility of the class number of any number field helps understand the arithmetic nature of the class group. The problem of constructing explicit infinite families of number fields whose class number has prescribed divisibility properties has attracted considerable attention over the years. For quadratic fields, there are many interesting results related to this divisibility problem of quadratic fields, and the reader is referred to \cite{ACC1955, BH2024, CHYP2018, CH2023, H2021, H2022} for more details. 

The study of the divisibility of the class number of cyclotomic fields, in particular its maximal real subfields, has also attracted the attention of many researchers. In this context, Ankeny, Chowla, and Hasse \cite{ACH} established that for certain prime $p$, the class number of the maximal real subfield of $\mathbb{Q}(\zeta_{p})$, where $\zeta_{p}$ is a primitive $p$-th root of unity, is strictly greater than one. Subsequent developments and related results on class numbers of cyclotomic fields and quadratic fields have been surveyed by Bhand and Murty \cite{ABRM}. In another relevant work, Hoque and Saikia \cite{HKS} proved that the class group of the maximal real subfield of the cyclotomic field $\mathbb{Q}(\zeta_{4m})$ for certain integers $m$ is non-trivial. Recently, Chakraborty and Hoque \cite{CH2025} constructed four infinite families of maximal real subfields of cyclotomic fields whose class number is strictly greater than $1$. They showed that the class group of the maximal real subfields of the cycltomic fields $\mathbb{Q}(\zeta_{4m})$, for each square-free integer $m$ of the form $\{14(2n+1)\}^{2}+1, \{3(2n+1)\}^{2}+1, \{6(2n+1)\}^{2}-2$ with $n\geq 1$, is nontrivial. They also proved that $3$ divides the class number of the maximal real subfield of the cyclotomic field $\mathbb{Q}_(\zeta_{m})$ for square-free integer $m=3(4\times3^{n}-1)$, where $n$ is a positive integer with $n\equiv 0\pmod{3}$.

In the year $2000$, Kishi and Miyake \cite{MI} parameterized certain quadratic fields whose class number is divisible by $3$. In this classification, they chose suitable integers $s$ and $t$ to construct irreducible cubic polynomials whose discriminants determine quadratic fields. If the corresponding cubic field turns out to be unramified, then by Hilbert class field theory, the class group of the quadratic field has an element of order $3$. In a recent work, Banerjee et al. \cite{AC} used this parameterization technique and proved that $3$ divides the class number of $\mathbb{Q}(\sqrt{m})$ for $m=216000x^{3}+457200x^{3}+322580x+75866,~ 432y^{3}+1080y^{2}+900y+223,~ 40500k^{3}+89100k^{2}+65340k+16215$, where $x,y$ and $k$ are any positive integers. Motivated by the above developments, we revisit the problem of the $3-$divisibility of class numbers of the real quadratic fields and the maximal real subfields of cyclotomic fields. First, we will fix the notations, and then we will go into our work. The following notations will be in use throughout. 
\begin{align*}
m>1 &: \text{ an integer} .\\
\zeta_{m} &: \text{ A primitive $m$-th root of unity }.\\
\mathbb{Q}(\sqrt{m})&: \text{ A real quadratic extension of $\mathbb{Q}$ }.\\ 
h(m)&:  \text{The class number of $\mathbb{Q}(\sqrt{m})$ }.\\
h_{4m} &: \text{The class number of the cyclotomic field $\mathbb{Q}(\zeta_{4m})$ }.\\
h^{+}_{4m} &: \text{ The class number of the maximal real subfield $\mathbb{Q}(\zeta_{4m}+\zeta^{-1}_{4m})$ of} ~\mathbb{Q}(\zeta_{4m}).\\
\Delta_{f} &: \text{ Discriminant of the polynomial $f$}.\\
(a,b) &: \text{$gcd$ of two integers $a$ and $b$}. 
\end{align*}
In this manuscript, the goal is to show $3$ divides $h_{4m}$ for the following choices of $m$: 
\begin{equation*}
\tag{1.1}\label{family}
 m= \begin{cases} 
  14580k^{3}+4860k^{2}+540k-655 & \text{with } k\in \mathbb{N} \text{ and } k\not\equiv 1\pmod{5}, \\
  864n^{3}+432n^{2}+72n-23 & \text{with } n\in\mathbb{N}, \\ 
  3000564\ell^{3}+5000940\ell^{2}+2778300\ell+507297 & \text{with } \ell\in \mathbb{N}\setminus2\mathbb{N} \text{ and } \ell\not\equiv 1\pmod{7}, \\
  324r^{3}+324r^{2}+108r-231 & \textit{with } r>1 \text{ is an integer}. 
\end{cases}
\end{equation*}

The paper is organized as follows. In \S 2, we recall the necessary preliminaries, including the Kishi-Miyake criterion and Yamaguchi's result. Section \S 3 contains four auxiliary lemmas, which will play a crucial role in establishing our results. In \S 4, we prove main results of this manuscript. To prove these results, we first show that $3$ divides $h(m)$ by using the precise results derived in \S 3, then we use a result of Yamaguchi \cite{YA} to show that $3$ divides $h^{+}_{4m}$ and thus $3$ divides $h_{4m}$. The resulting cubic polynomials, discriminants, and associated quadratic fields are different from those arising in the earlier constructions. Consequently, our results complement the existing literature by enlarging the collection of explicit infinite families of cyclotomic fields with class numbers divisible by $3$.

The concluding section \S 5 contains computational evidence (via tables) corroborating our results. We have used PARI/GP version 2.17.3 \cite{pari} for performing the computations.

\section{Preliminaries}\label{sec: 2} 

The following result of Kishi and Miyake \cite{MI} helps construct an unramified, cyclic and cubic extension of a given quadratic field such that $3$ divides the class number of this quadratic field.

\begin{thma}\label{t1}
    Let $s$ and $t$ be coprime integers and define 
    \begin{equation}
        f(\theta)=\theta^{3} - st\theta - s^{2}.
    \end{equation}
    Assume that $f(\theta)$ is irreducible over $\mathbb{Q}$ and that its discriminant $\Delta_f$ is not a perfect square. Suppose that one of the following conditions is satisfied:\\
\hspace*{2.5cm}
$(2.A.1)$ \hspace*{0.5cm}$3\nmid t$ ; \\
\hspace*{2.5cm}
$(2.A.2)$ \hspace*{0.5cm}$3\mid t, st\not\equiv 3\pmod{9}, s\equiv t\pm 1\pmod{9}$; \\
\hspace*{2.5cm}
$(2.A.3)$ \hspace*{0.5cm}$3\mid t, st\equiv 3\pmod{9}, s\equiv t\pm 1\pmod{27}$. \\ Then the normal closure of $\mathbb{Q}(\beta)$, where $\beta$ is a root of $f$, is a cyclic cubic extension which is unramified over the quadratic field $\mathbb{K}=\mathbb{Q}(\sqrt{\Delta_{f}})$, in particular, $3$ divides the class number of $\mathbb{K}$. Conversely, every quadratic field whose class number is divisible by $3$ admits such an unramified cyclic cubic extension arising from suitable choices of $s$ and $t$.
\end{thma}

The following result of Yamaguchi \cite{YA} gives a relation between $h(m)$ and $h^{+}_{4m}$.

\begin{thma}\label{t2}
          If $\phi\left(m\right)>4$, where $\phi$ is the Euler's totient function and $m$ is a positive integer, then $h(m)$ divides $h^{+}_{4m}$.
\end{thma}

\section{Few Lemmas}
We begin by proving four lemmas to establish that $m$ is not a perfect square in $\mathbb{Z}$ corresponding to the four families considered in \eqref{family}, which will be useful in proving the main results in \S 4.

\begin{lemma}\label{first theorem}
   Let $m=14580k^{3}+4860k^{2}+540k-655$, where $k$ is a positive integer such that $k\not\equiv 1\pmod{5}$. Then $m$ is not a perfect square in $\mathbb{Z}$. 
\end{lemma}
\begin{proof}
   First, observe that each coefficient of $m$ is divisible by $5$. Hence 
   $$
   m\equiv 0\pmod{5}.
   $$
If possible, let $m=r^2$ for some integer $r$. Then $5\mid r^2$, and therefore $5\mid r$, which implies 
$$
25\mid r^2=m.
$$
Thus, a necessary condition for $m$ to be a perfect square is  
$$
m\equiv 0\pmod{25}.
$$
Reducing the coefficients modulo $25$ gives, 
$$
m\equiv 5(k^3+2k^2+3k-1)\pmod{25}.
$$
It follows that
$$
m\equiv 0\pmod{25} \iff k^3+2k^2+3k-1 \equiv 0 \pmod {5}.
$$
A simple calculation modulo $5$ verifies that 
$$
k^3+2k^2+3k-1\equiv 0 \pmod {5} \iff k\equiv 1\pmod{5}.
$$
This is a contradiction to the hypothesis that $k\not\equiv 1\pmod{5}$. Thus, the result. 
\end{proof}

\begin{lemma}\label{second theorem}
 Let $m=864n^{3}+432n^{2}+72n-23$, where $n$ is any positive integer. Then $m$ is not a perfect square in $\mathbb{Z}$.
\end{lemma}
\begin{proof}
Note that, $m$ can be rewritten as
\begin{equation*}
m=4(6n+1)^{3}-27.
\end{equation*}
  If possible, let $m$ be a perfect square, i.e.
\begin{equation}\label{cp1}
y^{2}=4(6n+1)^{3}-27 
\end{equation}
for some $y$ in $\mathbb{Z}$. Let $x=6n+1$ and multiplying \eqref{cp1} by $16$, we get 
\begin{equation*}
(4y)^{2}=(4x)^{3}-432.
\end{equation*}
Further, substituting $Y=4y$ and $X=4x$, we get  
\begin{equation}\label{cp3}
Y^{2}=X^{3}-432
\end{equation} 
Here \eqref{cp3} represents a well known elliptic curve over $\mathbb{Q}$ which has rank $0$ and the torsion subgroup is $\{\mathcal{O},(12,\pm 36)\}$ as discussed in \cite{LCW}, where $\mathcal{O}$ represents the point at infinity. 

Thus $X=4x=12$ implies that $n=1/3$, which is a contradiction as $n$ is an integer. Thus, no such $y$ exists, and hence $m$ is not a perfect square.  
\end{proof}

\begin{lemma}\label{third theorem} Let $m=3000564\ell^{3}+5000940\ell^{2}+2778300\ell+507297$ where $\ell$ is positive integer. If $\ell$ is odd, then $m$ is not a perfect square.  
\end{lemma}
\begin{proof}
We reduce the expression $m$ modulo $8$ and get, 
\begin{equation}
    m\equiv 4\ell^{3}+4\ell^{2}+4\ell+1\pmod{8}. \label{cp2}
\end{equation}
Since $\ell$ is odd,
$$
\ell^{2}\equiv 1\pmod{8}, \ell^{3}\equiv\ell\pmod{8}.
$$
Substituting these congruences in \eqref{cp2}, we get 
$$
m\equiv 4\ell+4+4\ell+1=8\ell+5\equiv 5\pmod{8}.
$$
However, it is known that the quadratic residues modulo $8$ are $0,1$ and $4$. Therefore, $m$ cannot be a perfect square for any odd integer $\ell$.
\end{proof}

\begin{lemma}\label{fourth theorem}
Let $m=324r^{3}+324r^{2}+108r-231$, where $r>1$ is an integer. Then $m$ is not a perfect square.
\end{lemma}
\begin{proof}
    The proof proceeds along the same lines as in Lemma \ref{third theorem}. We use modulo $9$ in this case. 
\end{proof}    

\section{Class number of  certain cyclotomic fields}\label{sec: 4} 
In this section, we prove our main results. These results are concerned with the three divisibility of the class groups of certain cyclotomic fields. The fundamental steps to get these results are as follows. \\
$\bullet$ Consider $\mathbb{Q}(\sqrt{m})$ for $m$ as discussed in \S 3. \vspace*{2mm} \\
$\bullet$ Use Kishi and Miyake \cite{MI} to construct an unramified, cyclic and cubic extension of $\mathbb{Q}(\sqrt{m})$ and use it to conclude that $3$ divides $h(m)$. \vspace*{1mm} \vspace*{2mm} \\
$\bullet$ Use Yamaguchi \cite{YA} to conclude $3$ divides $h^{+}_{4m}$. \vspace*{2mm} \\
$\bullet$ Thus $3$ divides $h_{4m}$.

 \begin{theorem}\label{T1}
Let $m=14580k^{3}+4860k^{2}+540k-655$, where $k$ is a positive integer such that $k\not\equiv 1\pmod{5}$. Then $3$ divides $h_{4m}$. 
\end{theorem}
\begin{proof}
Let $s=5$ and $t=9k+1$. Clearly, $gcd(5, 9k+1)=1$ as $k\not\equiv 1\pmod{5}$. Let us set 
\begin{equation*}
f(\theta)=\theta^{3}-5(9k+1)\theta-25.
\end{equation*}
 We claim that $f$ is irreducible over $\mathbb{Q}$. Since $f$ is cubic, it suffices to prove that it has no rational root, because every reducible cubic polynomial over $\mathbb{Q}$ has a linear factor and hence a rational root.  Now, by the rational root test, the only possible rational roots of $f(\theta)$ are $\pm 1, \pm5$ and $\pm 25$. Direct calculations for $k\geq 1$ entails,
\begin{equation*}
f(1)= -45k-29\neq 0,
\end{equation*}
\begin{equation*}
f(-1)= 45k-21\neq 0,
\end{equation*}
\begin{equation*}
f(5)= 75-225k\neq 0,
\end{equation*}
\begin{equation*}
f(-5)= 225k-125\neq 0,
\end{equation*}
\begin{equation*}
f(25)= 15475-1125k\neq 0,
\end{equation*}
\begin{equation*}
 f(-25)= 1125k-15525\neq 0.
\end{equation*}
Therefore, $f$ has no rational root in $\mathbb{Q}$. Since $f$ is cubic, it follows that $f$ is irreducible over $\mathbb{Q}$. Furthermore, the discriminant of $f$,
\begin{equation*}
\Delta_{f}=4(5(9k+1))^{3}-27(-25)^{2},
\end{equation*}
can be written  as $(5)^2m$, where $m=14580k^{3}+4860k^{2}+540k-655$. Now lemma \ref{first theorem} confirms that $\Delta_{f}$ is not a perfect square. 

As $3\nmid t$, $f$ is irreducible over $\mathbb{Q}$ and $\Delta_{f}$ is not a perfect square, by Theorem \ref{t1}, we get that $3$ divides $h(m)$. Now we are in a position to apply Theorem \ref{t2} (since  $\phi(m)>4 ~\forall~ m$) to conclude that $3$ divides $h^{+}_{4m}$ and hence $3$ divides $h_{4m}$.
\end{proof}    
\begin{theorem}\label{T2}
Let $m=864n^{3}+432n^{2}+72n-23$, where $n$ is any positive integer. Then $3$ divides $h_{4m}$.
\end{theorem}
\begin{proof}
In this case, let $s=1$ and $t=6n+1$. Thus $\gcd(s,t)=1$ and we set:
\begin{equation*}
f(\theta):= \theta^{3}-(6n+1)\theta-1.
\end{equation*}
The discriminant of $f$ is 
\begin{equation*}
\Delta_{f}=4(6n+1)^{3}-27= 864n^{3}+432n^{2}+72n-23.
\end{equation*}
Clearly $\bar{f}(\theta)=\theta^{3}-\theta-1$ obtained by reducing $f$ modulo $3$, is irreducible over $\mathbb{Z}_{3}$ and thus $f$ is irreducible over $\mathbb{Q}$. Furthermore, by Lemma \ref{second theorem}, $\Delta_{f}$ is not a perfect square.

As before $3\nmid t$ and the application of Theorem \ref{t1} entails $3$ divides $h(m)$. Now by using Theorem \ref{t2}, we get $3$ divides $h_{4m}$ via the $3$-divisibility of $h^{+}_{4m}$. 

\end{proof}
    
\begin{theorem}\label{T3}
Let $m=3000564\ell^{3}+5000940\ell^{2}+2778300\ell+507297$, where $\ell$ is an odd  positive integer such that $\ell \not\equiv 1\pmod{7}$. Then $3$ divides $h_{4m}$. 
\end{theorem}
\begin{proof}
Let $s=7$ and $t=27\ell+15$. Clearly, $s$ and $t$ thus chosen are co-prime to each other as $\ell\not\equiv 1\pmod{7}$. We consider the polynomial 
\begin{equation*}
f(\theta)=\theta^{3}-7(27\ell+15)\theta-49.
\end{equation*}
 By the rational root test, the only possible rational roots of $f(\theta)$ are $\pm1,~ \pm7,~\pm49$. By direct substitution, none of these values is a root of $f(\theta)$ for any $\ell\geq 1$. Therefore,  $f(\theta)$ has no rational root, and since $f(\theta)$ is a cubic polynomial, $f(\theta)$ is irreducible over $\mathbb{Q}$. The discriminant of $f$ is
\begin{equation*}
\Delta_{f}=4(7(27\ell+15))^{3}-27(-49)^{2},
\end{equation*}
which can be simplified as $(3)^2m$, where $m=3000564\ell^{3}+5000940\ell^{2}+2778300\ell+507297$. Moving forward, Lemma \ref{third theorem} gives that $m$ is not a perfect square in $\mathbb{Z}$ and so is $\Delta_{f}$. \\
In this case $3\mid t$, $st\not\equiv 3\pmod{9}$ and $s-t=-27\ell-8\equiv 1\pmod{9}$, and hence we can apply Theorem \ref{t1} to conclude that $3$ divides $h(m)$. As before, by the application of Theorem \ref{t2}, we conclude that $3$ divides $h_{4m}$.
\end{proof}

\begin{theorem}\label{T4}
Let $m=324r^{3}+324r^{2}+108r-231$, where $r>1$ is an integer. Then $3$ divides $h_{4m}$. 
\end{theorem}
\begin{proof}
Let $s=3$ and $t=3r+1$ and hence $(s,t)=1$. We set, 
\begin{equation*}
f(\theta)=\theta^{3}-3(3r+1)\theta-9,
\end{equation*}
and thus 
\begin{equation*}
\Delta_{f}=4(3(3r+1))^{3}-27(-9)^{2}.
\end{equation*}
As before $\Delta_{f}$ can be expressed as $(3)^2m$, where $m=324r^{3}+324r^{2}+108r-231$. By Lemma \ref{fourth theorem}, $\Delta_{f}$ is not a perfect square in $\mathbb{Z}$. By the rational root test, the only possible rational roots are $\pm1,~\pm3,~\pm9$. A direct computation shows that none of these values is a root of $f(\theta)$ for any $r>1$. Hence $f(\theta)$ has no rational root and, therefore, $f(\theta)$ is irreducible over $\mathbb{Q}$ because $f(\theta)$ is a cubic polynomial over $\mathbb{Q}$. Since $3\nmid t$, following Theorem \ref{T1} the result follows. 
\end{proof}

\newpage
\section{Numerical Illustrations}
In this section, we present some numerical examples that support our results in §4. It is suﬃcient to compute the class numbers of each of the families of underlying real quadratic fields, i.e., $h(m)$'s. We compute $h(m)$ for the first twenty values of $m$ from each family. The Tables $1,2,3$ and $4$ provide numerical validation of Theorems \ref{T1} , \ref{T2} , \ref{T3} and \ref{T4} respectively. \vspace*{9mm}

\noindent
\begin{minipage}{0.48\textwidth}
\centering
\renewcommand{\arraystretch}{1}
\begin{tabular}{|>{\centering\arraybackslash}p{1cm}
                |>{\centering\arraybackslash}p{3cm}
                |>{\centering\arraybackslash}p{1cm}|}
\hline
$k$ & $m$ & $h(m)$ \\
\hline
2 & 136505 & 6 \\
3 & 438365 & 36 \\
4 & 1012385 & 12 \\
5 & 1946045 & 24 \\
7 & 5242205 & 108 \\
8 & 7779665 & 6 \\
9 & 11026685 & 18 \\
10 & 15070745 & 18 \\
12 & 25899905 & 72 \\
13 & 32859965 & 54 \\
14 & 40966985 & 6 \\
15 & 50308445 & 24 \\
17 & 73044605 & 24 \\
18 & 86614265 & 24 \\
19 & 101768285 & 6 \\
20 & 118594145 & 12 \\
22 & 157611305 & 12 \\
23 & 179977565 & 18 \\
24 & 204365585 & 48 \\
25 & 230862845 & 6 \\
\hline
\end{tabular}
\captionof{table}{Theorem \ref{T1}}
\end{minipage}\hfill
\renewcommand{\arraystretch}{1}
\begin{minipage}{0.48\textwidth}
\centering
\begin{tabular}{|>{\centering\arraybackslash}p{1cm}
                |>{\centering\arraybackslash}p{3cm}
                |>{\centering\arraybackslash}p{1cm}|}
\hline
$n$ & $m$ & $h(m)$ \\
\hline
1 & 1345 & 6 \\
2 & 8761 & 27 \\
3 & 27409 & 3 \\
4 & 62473 & 3 \\
5 & 119137 & 42 \\
6 & 202585 & 6 \\
7 & 318001 & 39 \\
8 & 470569 & 9 \\
9 & 665473 & 21 \\
10 & 907897 & 3 \\
11 & 1203025 & 18 \\
12 & 1556041 & 3 \\
13 & 1972129 & 3 \\
14 & 2456473 & 9 \\
15 & 3014257 & 3 \\
16 & 3650665 & 168 \\
17 & 4370881 & 18 \\
18 & 5180089 & 3 \\
19 & 6083473 & 3 \\
20 & 7086217 & 9 \\
\hline
\end{tabular}
\captionof{table}{Theorem \ref{T2}}
\end{minipage}

\noindent
\begin{minipage}{0.48\textwidth}
\centering
\begin{tabular}{|>{\centering\arraybackslash}p{1cm}
                |>{\centering\arraybackslash}p{3cm}
                |>{\centering\arraybackslash}p{1cm}|}
\hline
$\ell$ & $m$ & $h(m)$ \\
\hline
3 & 134865885 & 48 \\
5 & 514492797 & 12 \\
7 & 1294194909 & 18 \\
9 & 2617999293 & 12 \\
11 & 4629933021 & 84 \\
13 & 7474023165 & 24 \\
17 & 16234780989 & 96 \\
19 & 22439502813 & 72 \\
21 & 30052489341 & 234 \\
23 & 39217767645 & 48 \\
25 & 50079364797 & 12 \\
27 & 62781307869 & 6 \\
31 & 94282340061 & 672 \\
33 & 113369483325 & 288 \\
35 & 134873080797 & 24 \\
37 & 158937159549 & 72 \\
39 & 185705746653 & 168 \\
41 & 215322869181 & 24 \\
45 & 283678828797 & 1296 \\
47 & 322705720029 & 192 \\
\hline
\end{tabular}
\captionof{table}{Theorem \ref{T3}}
\end{minipage}\hfill
\begin{minipage}{0.48\textwidth}
\centering
\begin{tabular}{|>{\centering\arraybackslash}p{1cm}
                |>{\centering\arraybackslash}p{3cm}
                |>{\centering\arraybackslash}p{1cm}|}
\hline
$r$ & $m$ & $h(m)$ \\
\hline
2 & 3873 & 3 \\
3 & 11757 & 3 \\
4 & 26121 & 3 \\
5 & 48909 & 12 \\
6 & 82065 & 18 \\
7 & 127533 & 12 \\
8 & 187257 & 12 \\
9 & 263181 & 24 \\
10 & 357249 & 3 \\
11 & 471405 & 12 \\
12 & 607593 & 12 \\
13 & 767757 & 3 \\
14 & 953841 & 12 \\
15 & 1167789 & 72 \\
16 & 1411545 & 12 \\
17 & 1687053 & 9 \\
18 & 1996257 & 33 \\
19 & 2341101 & 24 \\
20 & 2723529 & 3 \\
21 & 3145485 & 24 \\
\hline
\end{tabular}
\captionof{table}{Theorem \ref{T4}}
\end{minipage}

\vspace*{0.7cm}
\textit{Remark:} The condition $k\not\equiv 1\pmod{5}$ in Theorem \ref{T1} is required in our proof. However, our computation suggests that for $k=1,6,11$, and $16$, we obtain $h(m)=6$, even though $k\equiv 1\pmod{5}$. At present, we do not have a theoretical proof that the congruence condition can be removed, and therefore we leave this as an open question for future investigation.

\section*{Acknowledgments}
The authors are thankful to Kalyan Banerjee and Azizul Hoque for many fruitful discussions throughout this work. The authors are indebted to the referee for his/her carefully reading the manuscript and suggesting many suitable changes. These have considerably enhanced the presentation of the manuscript. The authors also express their gratitude to SRM University-AP for providing the academic environment, resources, and institutional support to carry out this research.


\vspace*{2mm}
\begin{thebibliography}{10}
\vspace*{2mm}
\bibitem{ACC1955} N. C. Ankeny and S. Chowla, {\it On the divisibility of the class number of quadratic fields}, Pacific J. Math. {\bf 5} (1955), 321--324.
\bibitem{BH2024} K. Banerjee and A. Hoque, {\it Chow groups, pull back and class groups}, Monatsh. Math. {\bf 205} (2024), no. 3, 433--454.
\bibitem{AC} K. Banerjee, A. Chutia, and A. Hoque, {\it On the simultaneous $3 $-divisibility of class numbers of quadruples of real quadratic fields}, Hardy-Ramanujan J. {\bf 48}, 1-9.
\bibitem{CHYP2018} K. Chakraborty, A. Hoque, Y. Kishi, and P. P. Pandey, {\it Divisibility of the class numbers of imaginary quadratic fields}, J. Number Theory {\bf 185} (2018), 339--348.
\bibitem{CH2023} K. Chakraborty and A. Hoque, {\it Lehmer sequence approach to the divisibility of class numbers of imaginary quadratic fields}, Ramanujan J. {\bf 60} (2023), no. 4, 913--923.
\bibitem{CH2025} K. Chakraborty and A. Hoque, {\it On the plus parts of the class numbers of cyclotomic fields}, Chinese Annals Math. Series B {\bf 46} (2025), no. 2, 1--10. 
\bibitem{HKS} A. Hoque and H. K.  Saikia, {\it On the class number of the maximal real subfield of a cyclotomic field}, Quaest. Math. {\bf 39} (2016), no. 7, 889--894.

\bibitem {AK} K. Chakraborty and A. Hoque, {\it Pell-type equations and class number of the maximal real subfield of a cyclotomic field}, Ramanujan J. {\bf 46} (2018), no 3, 727-742.
\bibitem{H2021} A. Hoque, {\it On the exponents of class groups of some families of imaginary quadratic fields}, Mediterr. J. Math. {\bf 18} (2021), no. 4, Paper No. 153, 13 pp.
\bibitem{H2022} A. Hoque, {\it On a conjecture of Iizuka}, J. Number Theory {\bf 238} (2022), 464--473.
\bibitem{MI} Y. Kishi and K. Miyake, {\it Parametrization of the quadratic fields whose class number is divisible by three}, J. Number theory  {\bf 80} (2000), 209-217.

\bibitem{MA} D. A. Marcus, Number fields, 2nd ed., Springer, Switzerland, 2018.
\bibitem{YA} I. Yamaguchi, {\it  On the class-number of the maximal real subfield of a cyclotomic field}, J.Reine Angew. Math. {\bf 272} (1975), 217-220.

\bibitem{LCW} L. C. Washington, Elliptic Curves: Number Theory and Cryptography, 2nd ed., Chapman and Hall/CRC, 2008.

\bibitem{pari}
The PARI Group,
\textit{PARI/GP Computer Algebra System}, version 2.17.3, 2025.\\
Available at: \texttt{https://pari.math.u-bordeaux.fr/}

\bibitem{ACH} N.C. Ankeny, S. Chowla, and H. Hasse, {\it On the class number of the maximal real subfield of a cyclotomic field}, J. Reine Angew. Math., 217 (1965), 217-220. 

\bibitem{ABRM} A. Bhand and M. Ram Murty, {\it Class numbers of quadratic fields}, Hardy-Ramanujan J. 42 (2019), 17–25. 
                    
\end{thebibliography}
\end{document}